\documentclass[11pt]{amsart}

\usepackage{amssymb}
\usepackage{palatino}
\input amssym.def

\title{Divisors of Fourier coefficients of modular forms}

\author{Sanoli Gun and M. Ram Murty}

\address{Sanoli Gun\\ 
Institute for Mathematical Sciences, CIT Campus, Taramani, 
Chennai, 600 113, India}
\email{sanoli@imsc.res.in}

\address{M. Ram Murty\\  
Department of Mathematics, Queen's University, Kingston,
Ontario, K7L 3N6, Canada.}
\email{murty@mast.queensu.ca}

\thanks{ Research of the first author was partially 
supported by IMSc Number Theory grant.
Research of the second author was partially supported 
by an NSERC Discovery grant.}



\newcommand{\bZ}{{\mathbb Z}}
\newcommand{\bQ}{{\mathbb Q}}
\newcommand{\bR}{{\mathbb R}}
\newcommand{\bC}{{\mathbb C}}
\newcommand{\bF}{{\mathbb F}}

\newcommand{\C}{{\mathcal C}}
\renewcommand{\H}{{\mathcal H}}
\newcommand{\F}{{\mathbb F}}

\newtheorem{thm}{Theorem}

\newtheorem{lem}[thm]{Lemma}

\newtheorem{rmk}[thm]{Remark}

\newcommand{\lemref}[1]{Lemma~\ref{#1}}

\begin{document}

\begin{abstract}
Let $d(n)$ denote the number of divisors of $n$. 
In this paper, we study the average value of $d(a(p))$,
where $p$ is a prime and $a(p)$ is the $p$-th  Fourier coefficient 
of a normalized Hecke eigenform of weight $k \ge 2$ for $\Gamma_0(N)$ 
having rational integer Fourier coefficients. 
\end{abstract}

\maketitle   

\tableofcontents

\section{Introduction}

Throughout the paper, let $p, \ell$ be primes, 
$\H = \{z \in \C \mid \Im(z) > 0\}$
be the upper half plane. Also let
$N \ge 1$ be a natural number and $k \ge 2$ be an even integer.
Let $\pi(x)$ denote the usual prime counting function up to $x$.
Let $f$ be a normalized Hecke eigen cusp form of weight $k$ for $\Gamma_0(N)$ 
with Nebentypus $\chi$. Suppose that the Fourier expansion of 
$f$ at $i\infty$ is
$$
f(z) = \sum_{n \ge 1} a(n) q^n, 
$$
where $q = e^{2\pi i z}$. In this paper, we assume that $a(n)$ are 
rational integers.  The second author and  Kumar Murty \cite{MM}
considered the average value of  $\nu(a(n))$, where 
$\nu(n)$ is the number of distinct prime divisors of $n$.
In this paper, we investigate the sum $\sum_ {p \leq x} d(a(p))$,
where $d(n) = \sum_{\delta | n} 1$.
An essential ingredient in our work is a technique that can be traced 
back to van der Corput \cite{VC} who majorized the divisor 
function by short sums. This technique was later refined 
by many authors. We use a refinement due to Friedlander and 
Iwaniec \cite{FI}. Our result can be thought of as a
modular analogue of a result of Erd\"os \cite{PE}
who considered the asymptotics of 
$\sum_{n \leq x} d(F(n))$, where $F(x)$ is an
irreducible polynomial with integral coefficients.
Finding average value of divisors
of arithmetic functions has a long history. Some of the relevant
papers in this direction are \cite{PE}, \cite{FI} and \cite{LP}.

\smallskip

\section{Preliminaries and Statement of the result}

\smallskip

For an integer $\delta \ge 1$ and $x \in \bR$, set
\begin{eqnarray*}
\pi^{*} (x, \delta) &=& \# \{ p \leq x \mid  a(p) 
\equiv 0\!\!\! \pmod{\delta} \} \\
\text { and}\phantom{m} \pi (x, \delta) &=& 
\# \{ p \leq x \mid  a(p) \ne 0, 
~a(p) \equiv 0\!\!\! \pmod{\delta} \}.
\end{eqnarray*}

As before, let $f(z)= \sum_{n\ge 1} a(n) q^n$ be a 
normalized Hecke eigenform of 
weight $k$ for $\Gamma_0(N)$
with Nebentypus $\chi$ and rational integer 
Fourier coefficients. For a 
prime $\ell$, let  $\bZ_\ell$ denote the ring
of $\ell$-adic integers and ${\rm G}:=  
{\rm Gal}({\overline\bQ}/{\bQ})$. 
By the work of Deligne \cite{PD} 
and as the Fourier 
coefficients of $f$ are integers,  
there is a continuous representation 
$$
\rho_\delta: {\rm G} \to {\rm GL_2}
\big(\prod_{\ell | \delta}\bZ_{\ell}\big)
$$
(where the product is over distinct prime divisors)
for any positive integer $\delta > 1$.
This representation is unramified outside 
the primes dividing $\delta N$ and for $p \nmid N\delta$,
$$
\text{ tr }\rho_\delta(\sigma_p) = a(p), 
\phantom{mm} \text{det }\rho_\delta(\sigma_p) = \chi(p) p^{k-1},
$$
where $\sigma_p$ is a Frobenius element of $p$ in $G$
and $\bZ$ is embedded diagonally in 
$\prod_{\ell | \delta}{\bZ}_{\ell}$. 
Denote by $\tilde\rho_\delta$ the reduction
modulo $\delta$ of $\rho_\delta$:
$$
\tilde\rho_\delta: {\rm G} \xrightarrow{\rho_\delta} 
{\rm GL_2}\big(\prod_{\ell | \delta}\bZ_{\ell}\big) 
\twoheadrightarrow  {\rm GL_2}\big(\bZ/\delta\big).
$$
Let ${\rm H}_\delta$ be the kernel of $\tilde\rho_{\delta}$, 
${\rm K}_\delta$ the sub field of 
$\overline\bQ$ fixed by ${\rm H}_\delta$
and ${\rm G}_\delta = {\rm Gal}({\rm K}_\delta/\bQ)$. 
Let ${\rm C}_\delta$ be the subset of 
${\tilde\rho_\delta}(G)$
consisting of elements of trace zero and let 
$h(\delta) = |{\rm C}_\delta|/ |{\rm G}_\delta|$.

The condition $a(p) \equiv 0\!\!\pmod{\delta}$, 
where $(p, \delta{\rm N}) = 1$ means that for any 
Frobenius element $\sigma_p$  of $p$, 
$\tilde{\rho_\delta}(\sigma_p) \in {\rm C}_\delta$. Hence  
by the Chebotarev density theorem applied to 
${\rm K}_\delta/ \bQ$, we have
$$
\pi^*(x,\delta) \sim  \frac{|{\rm C}_\delta|}
{|{\rm G}_\delta|} \pi(x) = h(\delta) \pi(x).
$$
As ${\rm C}_\delta$ contains the image 
of complex conjugation, it is non-empty.
Note that $K_{\ell_1^{n_1}} \cap K_{\ell_2^{n_2}} = \bQ$
for distinct primes $\ell_1,\ell_2$ and natural numbers $n_1, n_2$.
This implies that $h(\delta) = \prod_{\ell^n || \delta} h(\ell^n)$,
where $\ell^n || \delta$ means that $\ell^n | \delta$
and $\ell^{n+1} \nmid \delta$. 

Now suppose that the Generalized Riemann 
Hypothesis (GRH) i.e. the Riemann 
Hypothesis for all Artin $L$-series is true.
Then by the works of Lagarias and Odlyzko \cite{LO}, 
one can show that
$$
\pi^*(x, \delta) = h(\delta) \pi(x) + 
O\left(\delta^3 x^{1/2} \log(\delta {\rm N}x)\right).
$$ 
An improved error term is available provided
one also assumes the Artin holomorphy 
conjecture as proved by M. R. Murty, V. K. Murty 
and N. Saradha \cite{MMS}.
Moreover, if we define
$$
{\rm Z}(x) = \{ p \le x \mid a(p) = 0 \}
$$
then as mentioned in \cite{MM} one can show 
the following lemma from the works 
of Ribet \cite{KR} and Serre \cite{JP};
\begin{lem}\label{lemma2}
Suppose that $f$ does not have 
complex multiplication. Then 
${\rm Z}(x) \ll x/ (\log x)^{3/2 - \epsilon}$ for 
all $\epsilon > 0$.
Further, suppose that GRH is true.
Then ${\rm Z}(x) \ll x^{3/4}$. 
\end{lem}
If $f$ has complex multiplication, 
then ${\rm Z}(x) \sim \frac{1}{2}\pi(x)$.
Now suppose that GRH is true. Then 
as noted by the second author 
and Kumar Murty \cite{MM}, one has: 
\begin{lem}\label{lemma3}
Suppose that $f$ does not have complex 
multiplication and GRH is true. Then for $x \ge 2$,
$$
\pi(x,\delta) = h(\delta)\pi(x) + 
O\left( \delta^3 x^{1/2} \log(\delta{\rm N} x)\right) 
+ O(x^{3/4}). 
$$ 
\end{lem}
Also it follows from the works of Carayol \cite{HC}, Momose \cite{FM},
Ribet \cite{KR1, KR2}, Serre \cite{JP1} and Swinnerton-Dyer \cite{SD} that 
for $\ell$ sufficiently large,
$$
{\rm T}_{\ell}: =  \text{ Im }\rho_{\ell} 
= \left\{ \gamma \in {\rm GL}_2(\bF_{\ell})
 \mid \det \gamma \in (\bF_{\ell}^{\times})^{k-1} \right\} .
$$
In this paper, we prove:
\begin{thm}\label{thmain}
Assume that GRH is true. Also, assume that $f$ 
is a normalized Hecke eigen cusp 
form of weight $k$ for $\Gamma_0(N)$
with rational integer Fourier 
coefficients $\{a(n)\}$. Moreover, suppose that
$f$ does not have complex multiplication. We have
\begin{eqnarray}\label{main}
x \ll \sum_{p \leq x \atop a(p) \ne 0} 
d(a(p)) \ll x(\log x)^{A},
\end{eqnarray}
where $A$ is an absolute constant which depends on $f$.
\end{thm}

\begin{rmk}
It is worth noting that above theorem is true unconditionally
when $f$ is a normalized Hecke eigen cusp form of $k=2$. 
Indeed, the estimate in \lemref{lemma3} is unconditional
if one is considering the case $k=2$.  For in this case, the modular
form corresponds to an elliptic curve by a 
celebrated theorem of Wiles \cite{wiles}
and subsequent work of Breuil, Conrad, Diamond and Taylor \cite{BCDT}.
With this theorem in hand, the primes enumerated by ${\rm Z}(x)$
are precisely the supersingular primes.  Indeed, based on a suggestion
of the second author, Elkies (see p. 25 of \cite{elkies1} 
and \cite{elkies2}) has shown unconditionally that the number
of supersingular primes is $O(x^{3/4})$. Thus, for $k=2$, we can dispense
with the GRH in \lemref{lemma3}.
\end{rmk}

In order to prove the theorem, the 
following lemmas play an important role.
The first one was proved by Friedlander 
and Iwaniec \cite{FI} and for
the proof of the second lemma we use Rankin's trick. 
But as Rankin \cite{RR} points out, it should really
be called Ingham's trick since Ingham told Rankin about it.
\begin{lem}\label{lem5}
Let $m, r \ge 2$ and $n\ge 1$. Then
$$ 
d_r(n) \le \sum_{\delta|n \atop \delta \le n^{1/m}} 
(2d(\delta))^{(r-1)\frac{m \log m}{\log 2}},
$$
where 
$$
d_r(n)= \sum_{n_1 \cdots n_r = 
n \atop n_1, \cdots, n_r \ge 1} 1.
$$
\end{lem}

\smallskip
\noindent
{\bf Proof.}$~~~$
See \cite{FI}.

\begin{lem}\label{lem7}
Suppose $b(n) \geq 0$ for $n\geq 1$ and 
$$
D(s):= \sum_{n=1}^{\infty} \frac{b(n)}{n^s}
$$
converges for $s \in \bC$ with $\Re(s)>t \geq 0$. 
Then 
$$
\sum_{n \le x} b(n) \leq x^u~D(u)
$$
for any $u, x \in \bR$ with $u>t$ and $x \geq 1$.
\end{lem}
\begin{proof}
Note that for any real number $u >t$, we have
$$
\sum_{n \le x} b(n)  ~\leq~ \sum_{n=1}^{\infty} 
b(n) \left(\frac{x}{n}\right)^u
~\leq~  x^u D(u).
$$
\end{proof}
In most applications, we choose $u = t + (1/\log x) $.

\smallskip

\section{A group theoretic estimate}

\smallskip

For an odd prime $\ell$, let $\rm B_{\ell} := 
{\rm GL }_2(\F_{\ell})$ and
$$
\rm A_{\ell} := \left\{ \gamma \in {\rm B_{\ell}} 
\mid~  \text{ tr } \gamma = 0 \right\}.
$$
The conjugacy classes of $\rm B_{\ell}$ are one 
of the following four types:
\begin{eqnarray*}
\alpha_a := \Big({a \ \ \ 0 \atop 0\ \ \  a}\Big),  && \phantom{m}
\alpha_b := \Big({a \ \ \ 1 \atop 0\ \ \ a }\Big),    \\
\alpha_{a,\delta} := \Big({a \ \ \ 0 \atop 
0\ \ \ \delta}\Big) ~~ a \ne \delta, && 
\beta_{a,b}:= \Big({a \ \ \ \epsilon b \atop  
b \ \ \ \ a }\Big)~~ b \ne 0, 
\end{eqnarray*}
where $a,b, \delta \in \F_{\ell}^{\times}$ and  
$\{1, \sqrt{\epsilon} \}$ is a basis 
for $\F_{\ell}^2$ over $\F_{\ell}$. The number of 
elements in these classes are $1,  \ell^2-1, \ell^2 + \ell$ 
and $\ell^2 - \ell$ respectively
(see Fulton and Harris \cite{FH}, page 68 for details).
Hence the elements  of $\rm A_{\ell}$ 
come from the conjugacy classes 
$\Big({a \ \ \ \ 0 \atop 0 \ \ -a}\Big) $ and 
$\Big({0 \ \ \epsilon b \atop  b \ \ \   0 }\Big)$, 
where $a, b \in \F_{\ell}^{\times}$. 
Further, the elements 
$\Big({  a \ \ \ \ 0 \atop 0 \ \ -a}\Big )$ and 
$\Big({ -a  \ \ \ 0 \atop \ \ 0 \ \ \  a}\Big )$
belong to the same class and the
elements $\Big({0 \ \ \epsilon b \atop  b \ \ \ 0 }\Big)$
and $\Big({0 \ \ -\epsilon b \atop  b \ \ \ \ \ 0 }\Big)$
belong to the same class.
Therefore 
$$
|\rm A_{\ell}| =  [(\ell^2 + \ell) (\ell -1) 
+  (\ell^2 - \ell) (\ell -1)]/2 = {\ell^2}(\ell -1).
$$
Also
$|\rm B_{\ell}| =  |\{ \gamma \in {\rm A_{\ell}} 
\mid ~{\rm det}{\gamma}\in \F_{\ell}^{\times} \}|
= (\ell^2 - 1)(\ell^2 - \ell)$. Further, we calculate the 
cardinality of the sets 
${\rm B}_{\ell^n} := {\rm GL}_2(\bZ/{\ell}^n\bZ)$  and 
$$
{\rm A}_{\ell^n} :=  \left\{ \gamma \in {\rm B}_{\ell^n} \mid~ 
\text{ tr }\gamma = 0 \right\} 
$$
for all $n \ge 1$.
Note that  any $\Big({a \ \ b \atop c  \ \  \delta}\Big) 
\in {\rm B}_{\ell}$
lifts to 
$$ 
\Big({a + \beta_1 \ell \ \ \ b + \beta_2 \ell \atop c 
+ \beta_3 \ell  \ \ \  \delta + \beta_4 \ell }\Big) \in {\rm B}_{\ell^n},
$$ 
where $1 \le \beta_1,\beta_2, \beta_3, \beta_4 \le \ell^{n-1}$, 
$\beta_i \in \bZ$ for all $1 \leq i \leq 4$. Also any 
$\Big({a \ \ \  \ b \atop c  \ \  -a}\Big) \in {\rm A}_{\ell}$
lifts to 
$$ \Big({a + \beta_1 \ell \ \ \ \ \ \ b + \beta_2 \ell \atop c
 + \beta_3 \ell  \ \  -a - \beta_1 \ell }\Big) \in {\rm A}_{\ell^n}
$$ 
for any choice of $\beta_1,\beta_2, \beta_3 \in \bZ$ with
$ 1 \le \beta_1,\beta_2, \beta_3 \le \ell^{n-1}$. 
Since the maps ${\rm B}_{\ell^n} \to  {\rm B}_{\ell}$ and  
${\rm A}_{\ell^n} \to {\rm A}_{\ell}$
are surjective, it is easy to see from the 
above observations that 
$|{\rm B}_{\ell^n}| = {\ell}^{4(n-1)} |{\rm B}_{\ell}|$
and $|{\rm A}_{\ell^n}| = {\ell}^{3(n-1)} |{\rm A}_{\ell}|$.

\smallskip

Next for an even integer $k>2$, we calculate the cardinality of the sets
\begin{eqnarray*}
C_{\ell}& :=& \left\{ \gamma \in {\rm B_{\ell}} 
~\mid~  \det \gamma \in (\bF_{\ell}^{\times})^{k-1} \right\} \\
\text { and } \phantom{mm}
D_{\ell} &:= & \left\{ \gamma \in C_{\ell} 
~|~ \text{ tr }\gamma = 0 \right\}.
\end{eqnarray*}
Writing gcd$(\ell -1, k-1):= d$, one has
\begin{eqnarray*}
C_{\ell} 
&=&  
\left\{ \gamma \in {\rm B_{\ell}} 
~\mid~  \det \gamma \in (\bF_{\ell}^{\times})^{d}\right\} \\
&=&  
\left\{ \gamma \in {\rm B_{\ell}} 
~\mid~  (\det \gamma)^{(\ell-1)/d}  \equiv 1 \!\!\!\pmod{\ell} \right\}. 
\end{eqnarray*}
Consider the surjective group homomorphism
$$
\phi: GL_2(\bF_{\ell})  \to (\bF_{\ell}^{\times})^{(\ell-1)/d}
$$
which sends $\gamma \mapsto (\det \gamma)^{(\ell-1)/d}$.
From the previous discussions, it is clear that 
$\text{ ker } \phi = C_{\ell}$. 
Hence
$$
|C_{\ell}| = \frac{| B_{\ell}|}{|(\bF_{\ell}^{\times})^{\frac{\ell-1}{d}}|}.
$$
But 
$$
|(\bF_{\ell}^{\times})^{\frac{\ell-1}{d}}|
= \frac{|\bF_{\ell}^{\times}|}
{\left|\left\{ x \in \bF_{\ell}^{\times} 
\mid x^{(\ell-1)/d} \equiv 1\!\!\!\pmod{\ell}\right\}\right|}
= \frac{\ell -1}{(\ell -1)/d} = d.
$$ 
Therefore
$$
|C_{\ell}| = \frac{(\ell^2 -1) (\ell^2 - \ell)}{d}.
$$
The elements of $D_{\ell}$ come from the conjugacy classes
$\alpha_{a,-a}$ 
with $-a^2 \in (\F_{\ell}^{\times})^{k-1}$ and 
$\beta_{0,a}$ with
$- \epsilon a^2 \in (\F_{\ell}^{\times})^{k-1}$. 
Let $g$ be the primitive root of $\F_{\ell}^{\times}$.
We would like to find the cardinality of the sets
\begin{eqnarray}\label{card}
 \{~  a  \mid -a^2 \equiv w^{k-1}\!\!\! \pmod{\ell} 
\text{ for some } w \in \F_{\ell}^{\times} \}
\end{eqnarray}
and
\begin{eqnarray}\label{card2}
 \{~  a  \mid -\epsilon a^2 \equiv w^{k-1} \!\!\!\pmod{\ell} 
\text{ for some } w \in \F_{\ell}^{\times} \} . 
\end{eqnarray}
Write $a = g^r, -1 = g^{\frac{\ell -1}{2}}$ and $w =  g^s$, 
where $0 \leq r,s \leq \ell -1$. Then the cardinality
of \eqref{card} is equal to the number of solutions $r$
for which 
\begin{equation}\label{sol}
\frac{\ell -1}{2} + 2r \equiv s(k-1) \!\!\!\pmod{\ell -1} ,
\end{equation}
where $0 \leq s \leq \ell -1$. This congruence has 
a solution $\{r_0, s_0\}$
if and only if $2r_0 \equiv -\frac{\ell -1}{2} \pmod{d}$.
Since $(2,d)=1$, the last congruence has a unique solution
in $r_0$.  Hence the number of $r$'s which are  
solutions of \eqref{sol} is $\frac{\ell - 1}{d}$. 
Note that if $a$ is in the set \eqref{card},  then
so is $-a$ and that
$\alpha_{a, -a} = \alpha_{-a,a}$.
Hence 
$$
|\{~  a  \mid -a^2 \equiv w^{k-1}\!\!\! \pmod{\ell} \text{ for some } 
w \in \F_{\ell}^{\times} \} | = \frac{\ell-1}{2d}. 
$$
Again writing $a = g^r,  -\epsilon = g^{t_0}$ and $w =  g^s$, 
where $0 \leq r,s \leq \ell -1$ and
solving the congruence
$$
t_0 + 2r \equiv s(k-1) \!\!\!\pmod{\ell -1} 
$$
we show that the cardinality of the 
set \eqref{card2} is $\frac{\ell -1}{2d}$.
Hence 
$$
|D_{\ell}| = \frac{\ell -1}{2d} (\ell^2  + \ell) + 
\frac{\ell -1}{2d} (\ell^2  - \ell)
  =  \frac{\ell^2 (\ell- 1)}{d} .
$$
Finally, we calculate for $n \geq 1$, the cardinality of the sets 
$$
C_{\ell^n} := \{ \gamma \in B_{\ell^n} \mid \gamma\!\!\!\! \pmod{\ell} \in C_{\ell} \}
$$
and 
$$
D_{\ell^n} = \{ \gamma \in C_{\ell^n} \mid \text{ tr }\gamma = 0  \}.
$$ 
Clearly, $|C_{\ell^n} | = {\ell}^{4(n-1)} |{\rm C}_{\ell}|$
and  $|D_{\ell^n}|  = {\ell}^{3(n-1)} |{\rm D}_{\ell}|$.

\smallskip

\section{Proof of the theorem}

\begin{proof}
Suppose that $f$ is a normalized Hecke eigenform of
weight $k$ for $\Gamma_0(\rm N)$ and $\delta$
is a large positive integer with the property
that if $p |\delta$, then $p \gg 1$. 
It follows from the previous two sections 
that for such $\delta$, we have
\begin{eqnarray}\label{estimate}
h(\delta) = \prod_{\ell^n||\delta} h(\ell^{n})
= \prod_{\ell^n||\delta} 
\frac{{\ell}^{3(n-1)}\ell}{{\ell}^{4(n-1)}(\ell^2 - 1)} 
= \prod_{\ell^n||\delta} 
\frac{\ell}{{\ell}^{n-1}(\ell^2 - 1)}.
\end{eqnarray}
Clearly when $\delta = \ell$ a prime, 
$h(\ell) \asymp \frac{1}{\ell}$ for sufficiently
large $\ell$. For a lower bound, note that 
\begin{eqnarray*}
\sum_{p \leq x \atop a(p) \ne 0}d(a(p)) 
~=~ \sum_{p \leq x \atop a(p) \ne 0} \sum_{\delta | a(p) } 1 
&\ge & \sum_{p \le x \atop a(p) \ne 0}
\sideset{}{}{\sum}_{\substack{\delta < x^{1/12}\atop 
a(p) \equiv 0\!\!\!\!\!\pmod{\delta}}} 1  \\
&\geq & {\sum}_{\substack{\delta < x^{1/12}}}^* \pi(x,\delta), \nonumber
\end{eqnarray*}
where $\sum^*$ varies over all those natural numbers $\delta$ whose prime divisors
are sufficiently large.  Hence by \lemref{lemma3}, we have
\begin{eqnarray*}
\sum_{p \leq x \atop a(p) \ne 0}d(a(p)) 
&\geq & \pi(x) {\sum}_{\substack{\delta < x^{1/12}}}^* h(\delta)
~+~  O(x^{1/2} {\sum}_{\substack{\delta < x^{1/12}}}^* \delta^3 \log\delta) 
~+~  O(x^{5/6}) \nonumber \\
&= & \pi(x) \sum_{\delta < x^{1/12} \atop \ell | \delta \iff \ell \gg 1} h(\delta) 
~+~ O( x^{5/6} \log x)  ~\gg ~  x.
\end{eqnarray*}
For an upper bound, we can use \lemref{lem5} to get 
\begin{eqnarray*}
\sum_{p \le x \atop a(p) \ne 0 } d(a(p)) 
&\ll& \sum_{ p\le x \atop a(p) \ne 0} 
\sum_{\delta|a(p) \atop \delta \le {|a(p)|^{1/m}}}
{d(\delta)}^{\frac{m \log m}{\log 2}}.
\end{eqnarray*}
We choose $m$ so that $m > 7k$.
Write $c=\frac{m \log m}{\log 2}$.  
As $|a(p)|  < 2 p^{k/2} $, we have
\begin{eqnarray*}
\sum_{p \le x \atop a(p) \ne 0} d(a(p)) & \ll  & 
\sum_{ p \le x \atop a(p) \ne 0}
\sum_{ \delta|a(p) \atop \delta < {x^{1/12}}} d(\delta)^{c} \\ 
& = & \sum_{\delta < x^{1/12}} d(\delta)^{c}~ \pi(x, \delta) \\
& = &  \sum_{ \delta < x^{1/12} } d(\delta)^{c} \left\{ h(\delta)\pi(x) 
+ O\left( \delta^3 x^{1/2} \log(\delta{\rm N} x)\right) 
+ O(x^{3/4})\right\} 
\end{eqnarray*}
by using \lemref{lemma3}.  
Note that when $\delta$
has small prime divisors, the
value of $h(\delta)$ is less than
the value of the right hand side of \eqref{estimate}.
Hence for an upper bound we can use the
right hand side of \eqref{estimate}
for all values of $\delta$.

Consider the Dirichlet series
\begin{eqnarray*}
{\rm F}(s):= \sum_{n \ge 1} \frac{d(n)^{c} h(n)}{n^s} 
= \zeta(s+1)^{2^c} g(s), 
\end{eqnarray*}
where $g(s)$ is analytic for $\Re(s) \geq 0$. Thus
by \lemref{lem7}, we have
$$
\sum_{n \leq z} d(n)^c h(n) \leq z^u F(u),
$$
for any real number $u > 0$. We choose $u = 1/ \log z$
so that 
$$ 
\sum_{n \leq z} d(n)^c h(n) \ll F\left(\frac{1}{\log z}\right)
$$
Since
$$
|\zeta(s)| \leq \frac{1}{s-1} +   L,
$$ 
where $L$ is an absolute constant, we see easily
$$ 
\sum_{n \leq z} d(n)^c h(n) \ll (\log z)^{2^c}.
$$
Again consider the Dirichlet series
\begin{eqnarray*}
{\rm G}(s):= \sum_{n \ge 1} \frac{d(n)^{c}}{n^s} 
= \zeta(s)^{2^c} g_1(s), 
\end{eqnarray*}
where $g_1(s)$ is analytic for $\Re(s) \geq 1$. 
Thus by \lemref{lem7}, we have
$$
\sum_{n \leq z} d(n)^c  \leq z^u G(u),
$$
for any real number $u > 1$. We choose $u = 1/ \log z + 1$
so that 
$$ 
\sum_{n \leq z} d(n)^c   \ll z G\left(\frac{1}{\log z} + 1 \right)
$$
and hence
$$ 
\sum_{n \leq z} d(n)^c   \ll z (\log z)^{2^c}.
$$
This implies that
\begin{eqnarray*}
\sum_{p \le x \atop a(p) \ne 0} d(a(p)) 
~\ll~   x (\log x)^{2^c -1}   +  O\left( x^{5/6} (\log x)^{2^c + 1}\right) 
~\ll~  x (\log x)^{2^c -1}.
\end{eqnarray*}
This completes the proof of the theorem.
\end{proof}

\section{Concluding remarks}

We hasten to remark that the full strength of the GRH is not essential
if one only wants an estimate of the form $x(\log x)^A$ for some $A$.
Indeed, if one assumes a quasi-GRH (that is, the assumption for any
given $\epsilon >0$,  the
Artin $L$-series have no zero in the region $\Re(s) > 1-\epsilon$),
then one can deduce a result of the form
$$\sum_{p \leq x, \atop a(p) \ne 0 } d(a(p)) \ll x (\log x)^A, $$
for some  $A$ depending on $\epsilon$.  This is not difficult to see.
Indeed, a version of \lemref{lemma2} with an estimate of the form $x^{1-\epsilon}$
can easily be deduced under such a hypothesis.  In addition, one can choose
$m$ appropriately in \lemref{lem5} so as to ensure that the subsequent sums in 
the proof of the main theorem can be reasonably estimated.  
It is also evident that the lower bound can be deduced from a version
of the Chebotarev density theorem derived from a quasi-GRH.  We leave the
details to the reader.  All of this analysis suggests the following
question.  Is it reasonable to expect that there exist constants $B$ and
 $v$ such that we have an asymptotic formula of the type
$$
\sum_{p \leq x} d(a(p)) ~\sim~  Bx(\log x)^v 
$$
as $x \to \infty$? Perhaps $v=0$.

\medskip
\noindent
{\bf Acknowledgments :}
The authors thank Kumar Murty and the referee for making some valuable suggestions
on an earlier draft of this paper.

\end{document}